\documentclass[letterpaper, 10 pt, conference]{ieeeconf}  

\IEEEoverridecommandlockouts                              
\usepackage{graphicx} 
\usepackage{subcaption}
\usepackage{url}
\usepackage{amsmath} 
\usepackage{amssymb}  
\usepackage{algorithm}
\usepackage[noend]{algorithmic}
\usepackage{xcolor} 

\def\ddefloop#1{\ifx\ddefloop#1\else\ddef{#1}\expandafter\ddefloop\fi}

\def\ddef#1{\expandafter\def\csname c#1\endcsname{\ensuremath{\mathcal{#1}}}}
\ddefloop ABCDEFGHIJKLMNOPQRSTUVWXYZ\ddefloop

\def\ddef#1{\expandafter\def\csname s#1\endcsname{\ensuremath{\mathsf{#1}}}}
\ddefloop ABCDEFGHIJKLMNOPQRSTUVWXYZ\ddefloop

\def\ddef#1{\expandafter\def\csname b#1\endcsname{\ensuremath{\mathbb{#1}}}}
\ddefloop ABCDEFGHIJKLMNOPQRSTUVWXYZ\ddefloop
    
\DeclareMathOperator*{\argmax}{arg\,max}

\newtheorem{thm}{Theorem}
\newtheorem{lem}[thm]{Lemma}
\newtheorem{prop}[thm]{Proposition}

\newtheorem{assum}[thm]{Assumption}
\newtheorem{defn}[thm]{Definition}

\title{\LARGE \bf Revisiting Regret Benchmarks in Online Non-Stochastic Control 

}

\author{Vijeth Hebbar and C\'edric Langbort 
\thanks{This work was supported by the \textit{ARO} MURI grant W911NF-20-0252 (76582 NSMUR).}
\thanks{All authors are with the Coordinated Science Lab, University of Illinois at Urbana–Champaign,          Urbana, IL 61801, USA.
        {\tt\small \{vhebbar2,langbort\}@illinois.edu}}%
}

\begin{document}

\maketitle
\thispagestyle{empty}
\pagestyle{empty}

\begin{abstract}
In the online non-stochastic control problem, an agent sequentially selects control inputs for a linear dynamical system when facing unknown and adversarially selected convex costs and disturbances. A common metric for evaluating control policies in this setting is policy regret, defined relative to the best-in-hindsight linear feedback controller. However, for general convex costs, this benchmark may be less meaningful since linear controllers can be highly suboptimal. To address this, we introduce an alternative, more suitable benchmark—the performance of the best fixed input. We show that this benchmark can be viewed as a natural extension of the standard benchmark used in online convex optimization and propose a novel online control algorithm that achieves sublinear regret with respect to this new benchmark. We also discuss the connections between our method and the original one proposed by Agarwal et al. in their seminal work introducing the online non-stochastic control problem, and compare the performance of both approaches through numerical simulations. 

\end{abstract}


\section{INTRODUCTION}
The interface between \emph{online learning} and \emph{control theory} has become an active area of research, driven by the need to make sequential decisions in dynamic and uncertain environments \cite{li2021online, shi_online_2020, nonhoff_online_2024,jiang_online_2025}. In this paper, we revisit a control problem that lies within this intersection and has garnered increased attention in recent years; the online \emph{non-stochastic} control problem \cite{hazan_introduction_2025}. Our contributions are twofold. First, we introduce a new notion of regret, based on a novel benchmark,  that provides a more meaningful metric for evaluating performance in this setting. Second, we propose an online control algorithm that achieves sublinear regret with respect to this benchmark.


In the online non-stochastic control problem, we are interested in controlling a linear dynamical system when faced with a sequence of \emph{unknown} and \emph{adversarially} picked convex cost functions and disturbances. Specifically, we consider the following 
linear time-invariant system
\begin{equation} \label{eq:LTI-sys-disturbed}
x_{t+1} = A x_t + B u_t + w_t, \quad t\geq 1,
\end{equation}
and are interested in designing a learning algorithm $\cA$ that picks inputs $\{u_t\}$ to minimize the cumulative cost\footnote{We consider input-independent costs for ease of exposition, but our methods can be easily extended to input-dependent costs as described in Appendix \ref{app:input-dependent}.}
\begin{flalign*}
    J(\cA) = \sum_{t=1}^T f_t(x_t).
\end{flalign*}
Crucially, the learner has to take action $u_t$ at time $t$ before any information about the cost function $f_t$ or the disturbance $w_t$ is revealed. Consequently, there is no way for the learner to compute the optimal controller. Instead, the standard metric of performance used in such scenarios is \emph{policy regret}, which compares the cumulative cost incurred by the algorithm against the performance of the best-in-hindsight policy from a certain policy class.





\paragraph{Background} This setup was first considered by Agarwal et al. \cite{agarwal_online_2019} and has since been extended in a multitude of directions -- consideration of strongly convex costs \cite{agarwal_logarithmic_2019}, partial observability of state \cite{simchowitz_improper_2020}, consideration of safety constraints \cite{li2021online,jiang_online_2025} 
and more \cite{cassel_rate-optimal_2022,cassel2020bandit}. In all these works, the policy class from which the best-in-hindsight controller is picked is the set of all \emph{stabilizing} linear feedback controllers $\cK$. Formally, the regret is defined as 
\begin{flalign} \label{eq:regret_constant_K}
    \cR_{\cK} (T) = \sum_{t=1}^T f_t(x_t) - \min_{K\in \cK} \sum_{t=1}^T f_t(x_t^K)
\end{flalign}
where $\{x_t^K\}_{t=1}^T$ is the state trajectory generated by the time-invariant linear controller $K$, i.e., $u_t = K x_t$ for all $t$.

The primary motivation for choosing this benchmark class is rooted in the classical linear quadratic regulator (LQR) problem: when costs are fixed quadratic functions and the horizon is infinite, the optimal controller is indeed a fixed linear feedback controller. Now, for a finite-horizon LQR problem with time-varying cost functions, such a fixed linear controller may perform poorly. However, the optimal controller for a finite horizon LQR is ``non-causal" in the sense that computing the controller $K_t$ at time $t$ requires the knowledge of future costs $\{Q_\tau,R_\tau\}_{\tau\geq t}$. Since this information is not available in an online learning setup, it is impractical and unfair to compare the learner's performance against the optimal (yet unattainable) benchmark. Consequently, a fixed linear feedback has become the standard benchmark in online LQR literature~\cite{cohen_online_2018}.
When considering general convex functions, this benchmark of a fixed linear feedback controller becomes substantially less meaningful. To clearly illustrate this, consider a simplified scenario where no disturbances act on the system ($\{w_t\} \equiv \mathbf{0}$), and suppose the cost functions $\{f_t\}$ all attain their minimum at some fixed non-zero state $x^*$.  
Any stabilizing linear feedback controller will quickly drive the system to the origin, leading to a cumulative cost close to $\sum_t f_t(0)$. In contrast, a controller that stabilizes the system at or near $x^*$ would yield significantly lower cumulative costs. The primary motivation for retaining linear feedback controllers as a benchmark is that we recover established regret guarantees when applying the developed approach to an LQR setup. 
However, prompted by the limitations of linear controller benchmarks when facing general convex costs, this paper seeks to explore alternative and more suitable benchmarks.

A natural option -- common in the online convex optimization (OCO) literature -- is the optima of the cumulative cost function. Mapping this to our setting, we could propose the following notion of regret: \vspace{-0.15cm}
\begin{flalign}
\cR_x(T) \triangleq \sum_{t=1}^T f_t(x_t) - \min_{x \in \cX} \sum_{t=1}^T f_t(x), \vspace{-0.15cm}\label{eq:regret_x_intro}
\end{flalign} 
where \( \cX \) denotes the set of steady-states achievable via constant inputs (formally defined in Definition~\ref{def:ss_set}). 
This regret definition is particularly appealing because it captures how well the learner performs relative to an idealized controller that instantly stabilizes the system at the optimal steady-state \( x^* \in \cX \). This is a powerful notion of regret since it is identical to the one in standard online convex optimization setups - where decisions $\{x_t\}$ are chosen independently from $\cX$ without any coupling across time. Thus, achieving low regret against this benchmark means that, despite operating under system dynamics, the learner performs almost as well as if the dynamics posed no constraints. 

However, directly using this benchmark can be challenging when the system is subject to disturbances. In such settings, stabilizing to a single point is not feasible since the disturbance continually pushes the system around. To address this, we work with an alternate regret definition that compares performance against the best constant input: \vspace{-0.15cm}
\begin{flalign}
\cR_u(T) \triangleq \sum_{t=1}^T f_t(x_t) - \min_{u \in \cU} \sum_{t=1}^T f_t(x_t^u), \vspace{-0.15cm} \label{eq:regret_u_intro}
\end{flalign}
where \( \{x_t^u\} \) is the trajectory induced by applying the fixed input \( u \) at all times. Crucially, we show that in the absence of disturbances, the two regret definitions in (\ref{eq:regret_x_intro}) and (\ref{eq:regret_u_intro}) are \emph{equivalent} in some sense (see Lemma~\ref{lem:equiv-regret}), thus offering a meaningful interpretation for $\cR_u$ as well.  However, this regret formulation remains meaningful even in the presence of disturbance. We employ both these regret notions in our work. 



\paragraph{Other Relevant Literature} A closely related line of work within the control literature considers \emph{dynamic regret}, which compares the performance of the learner against that of the optimal clairvoyant controller (i.e., a controller informed by complete knowledge of future disturbances and costs). This is, of course, the strongest benchmark against which any online control policy can be compared. However, unlike \emph{static} regret bounds (which we focus on in this work), dynamic regret bounds are typically instance-dependent, meaning they depend explicitly on the sequence of costs encountered. Moreover, these results often rely on additional assumptions. For example, Li et al. \cite{li_online_2019} assume access to a finite look-ahead window of upcoming cost functions and design a controller that guarantees dynamic regret bounds that scale with the total variation in the sequence of point-wise optima i.e. $\{\argmax f_t(x)\}$.  Nonhoff et al \cite{nonhoff_online_2024} also consider dynamic regret, but use a weaker benchmark — specifically, the sequence of point-wise optima restricted to the steady-state manifold (see Definition \ref{def:ss_set}). They also obtain regret bounds that depend on the total variation of these optima. Our focus in this work is to derive instance-independent bounds on the static regret of our problem. 



\paragraph{Roadmap} In Section \ref{sec:setup}, we formalize our problem setup and the notion(s) of regret we employ. Section \ref{sec:alg+regret} proposes a novel online control policy for online non-stochastic control and presents associated regret bounds. We first focus on a simplified setup (Section \ref{sec:undisturbed_sec}) where disturbances are absent from our system and develop an algorithmic approach for it. 
Then in Section \ref{sec:disturbed_sys}, we reintroduce the disturbances in our setup and show a key reduction of our problem to an equivalent disturbance-free formulation. Leveraging this reduction allows us to directly extend the previous solution approach to this more general scenario with disturbances. 


In Section \ref{sec:compare}, we explore connections and differences between our approach and the one taken by Agarwal et al. \cite{agarwal_online_2019} in solving the online non-stochastic control problem. Specifically, we compare the policy classes and the benchmarks employed in both works. Finally, in Section \ref{sec:simu} we present some numerical simulation results and observe that the proposed approach achieves lower cumulative costs that the method described by Agarwal et al. \cite{agarwal_online_2019}  .  

\section{Problem Setup} \label{sec:setup}




Let $x_t\in \bR^N$ denote the state of a system and let $u_t\in\cU\subset\bR^M$ denote the input provided by the learner at time $t$. The system is governed by the linear time-invariant dynamics outlined in (\ref{eq:LTI-sys-disturbed}), where $A$ and $B$ are real matrices of appropriate dimension and $w_t\in\cW\subset\bR^N$ denotes the disturbance picked by an adversary at time $t$. The learner faces a sequence of \emph{convex} cost functions $\{f_t\}_{t=1}^T$, also picked by an adversary. We denote the gradient $\nabla f_t(x_t)$ as $\delta_t$.  

In every time step $t$, the following sequence of events takes place: the learner first observes the current state $x_t$, then they pick an action $u_t$, then they incur cost $f_t(x_t)$ and receive the gradient $\delta_t$ as feedback, and finally, the system evolves to the next state according to (\ref{eq:LTI-sys-disturbed}). So, we are considering a setup where the learner receives \emph{gradient feedback}, and uses it to pick inputs that yield low cumulative cost.  


\subsection{Assumptions} \label{sec:assum}
Throughout this paper, we assume the following.

\begin{assum}[Strong Stability] There exist constants $\kappa>0$ and $\gamma \in (0,1]$,     and matrices $H$ and $J$ such that $A = H J H^{-1}$ with $\|J\| \leq 1- \gamma$. \label{ass:strong_stab}
\end{assum}

Here $\|\cdot\|$ denotes the spectral norm of a matrix. We also use the same notation to denote the $2$-norm of a vector, and the norm we employ will be clear from the context.

Note that Assumption \ref{ass:strong_stab} is not particularly restrictive.  Strong stability is merely a quantitative refinement of Schur stability and every Schur stable matrix is strongly stable for some $\gamma$ and $\kappa$. A more detailed discussion on this equivalence can be found in Cohen et al. \cite{cohen_online_2018}. Moreover, this assumption can be further relaxed by requiring only that the system $(A,B)$ is stabilizable. In this case, one could select a stabilizing feedback matrix $K$ such that the closed-loop system $\tilde{A} \triangleq A - BK$ satisfies the strong stability condition stated in Assumption \ref{ass:strong_stab}. The system can then be rewritten as
\[x_{t+1} = A x_t + B (-Kx + \tilde{u}_t) + w_t= \tilde{A} x_t + B\tilde{u}_t + w_t, \]
and our goal would be to design the control input $\tilde{u}_t$ to achieve the desired performance. We make Assumption \ref{ass:strong_stab} because it allows us to significantly improve the readability and clarity of our regret analysis, but as argued above, it is also without much loss of generality. We also make the following assumption:



\begin{assum} \label{ass:input-bd}
    The input set $\cU$ and disturbance set $\cW$ are bounded. 
\end{assum}

Note that strong stability of our system also implies that it is stable (spectral radius $\rho(A) = \rho(J) \leq \|J\| < 1$), which in turn implies BIBO stability. As a consequence, under Assumption \ref{ass:input-bd}, the state $x_t$ at any time $t$ remains bounded, and we will employ $D$ to denote the bound on $\|x_t\|$. Note that the value of $D$ may depend on the initial state $x_1$ and on the bounds on the sets $\cU$ and $\cW$. Finally, we make the following assumption:



\begin{assum}[Smoothness of costs] \label{ass:smooth-cost}
    \[\|\nabla f(x)\| \leq L D \quad \forall \; x \;\text{ S.T }\|x\| \leq D\] 
\end{assum}



\subsection{Notion of Regret}
We now formalize the performance metrics used to evaluate our online control policy. We begin by defining the set of steady-states achievable under constant inputs in the absence of disturbances:
\begin{defn}[Steady State Manifold] \label{def:ss_set}
    \[ \cX = \{x\ \in \bR^N \; | \; x = A x + B u, \quad u\in \cU\}.\]
\end{defn}
Note that this set is well defined under our strong-stability assumption (Assumption \ref{ass:strong_stab}), which ensures that $I-A$ is invertible. Now suppose that our learner picks the sequence of input $\{u_t\}_{t=1}^{T-1}$, generating a corresponding state trajectory $\{x_t\}_{t=11}^T$. Then, the \emph{steady-state regret} incurred by the learner is defined as  \vspace{-0.2cm}
\begin{flalign}
    \cR_{x}(T) \triangleq \sum_{t=1}^T f_t(x_t) - \min_{x\in \cX} \sum_{t=1}^T f_t(x). \vspace{-0.1cm} \label{eq:reg-defn-x}
\end{flalign}
As discussed in the introduction, while this notion of regret is quite meaningful and powerful, it may not be appropriate when disturbances act on the system. In practice, achieving this benchmark may be impossible even in the disturbance-free case and with complete knowledge of the cost functions because the initial state $x_1$ of the system might differ from the optimal state $x^*$. A more attainable benchmark considers the cumulative cost incurred by the best time-invariant input and accordingly, we can propose the \emph{constant input regret} as 
\begin{flalign}
    \cR_{u}(T) \triangleq \sum_{t=1}^T f_t(x_t) - \min_{u\in \cU} \sum_{t=1}^T f_t(x_t^u) \label{eq:reg-defn-u}
\end{flalign} 
where $\{x_t^u\}$ denote the state trajectory starting from $x^u_1$ under the time-invariant input $u_t=u$. However, as formalized in the Lemma \ref{lem:equiv-regret} below, any online control policy that performs well with respect to the regret definition in (\ref{eq:reg-defn-x}) also does so with respect to the one in (\ref{eq:reg-defn-u}).

\begin{lem} \label{lem:equiv-regret} Under the assumption that disturbances are absent i.e. $\{w_t\} \equiv \mathbf{0}$, we have \vspace{-0.2cm}
    \[\cR_u(T) - \cR_x(T) \leq \frac{2 \kappa LD^2 }{\gamma} \in o(T)\]
\end{lem}



We will use the regret definition in (\ref{eq:reg-defn-x}) when presenting the regret bounds in Section \ref{sec:undisturbed_sec} which considers the system without disturbances. Then in Section \ref{sec:disturbed_sys}, we extend our approach to also address the adversarial disturbances, and present regret bounds in terms of the definition in (\ref{eq:reg-defn-u}).  

\section{Learning Algorithm and Regret Bounds} \label{sec:alg+regret}
\subsection{System without disturbance}\label{sec:undisturbed_sec}
For the purpose of this section alone, we make the assumption that there are no disturbances acting on our system i.e. $\{w_t\} \equiv \mathbf{0}$. The online control policy we propose is presented in Algorithm \ref{alg:OLC}. 
This algorithm keeps track of a `target state' $z_t$, which is updated through an online projected gradient descent update in every step (Line \ref{lin:projection}). It then computes and returns the input that would stabilize the system at this target state (Line \ref{lin:invert_z}). \vspace{-0.2cm}
 \begin{algorithm}[H]
 \caption{Online Linear Control Algorithm}
 \begin{algorithmic}[1]
 \renewcommand{\algorithmicrequire}{\textbf{Input:}}
 \REQUIRE Step size $\eta$, Initial state $x_1$ \\
    \textit{Initialization:} $z_1 \in {\cX}$ \\
     \FOR{$t=1,\cdots,T-1$}
        \STATE Pick $u_t$ S.T. $Bu_t = (I-A)z_t$ \label{lin:invert_z} 
        \STATE Incur cost $f_t(x_t)$, Record gradient $\delta_t = \nabla f_t(x_t)$ \label{lin:feedback} 
        \STATE Update target state $z_{t+1} = \Pi_{\cX} \big(z_t-\eta \delta_t)$ \label{lin:projection}
        \STATE Update the state $x_{t+1} = A x_t + B u_t$ \label{lin:x_update}
     \ENDFOR
 \end{algorithmic} 
 \label{alg:OLC}
 \end{algorithm}
\vspace{-0.2cm}
We can state the following theorem bounding the regret incurred by our algorithm.
\begin{thm} \label{thm:undist_regret_bound}
    Picking inputs according to Algorithm \ref{alg:OLC} with $\eta = \frac{2\gamma}{L\sqrt{T(1+4\kappa^2)}}$ guarantees a regret bound of \vspace{-0.2cm}
    \begin{flalign*}
        \cR_x(T)  \leq \frac{2LD^2}{\gamma}\bigg(\sqrt{T(1+4\kappa^2)}+\kappa\bigg).  
    \end{flalign*} 
\end{thm}

\begin{proof}(Sketch) 
The convexity of costs gives us
\begin{flalign}
    \cR_x(T) & = \sum_{t=1}^T f_t(x_t) - f_t(x^*) \leq  \sum_{t=1}^T \delta_t^\intercal (x_t - x^*) \notag \\
    & = \sum_{t=1}^T \underbrace{\delta_t^\intercal (z_t - x^*)}_{(A)} + \underbrace{\delta_t^\intercal (x_t - z_t)}_{(B)} \label{eq:regret_analysis_1}
\end{flalign}
We will use two different analytic techniques to bound the two types of terms $(A)$ and $(B)$. Term $(A)$ captures how closely the sequence of target states $\{z_t\}_1^T$ generated by online gradient descent (Line \ref{lin:projection}) approximates the optimal state $x^*$. We can use standard online gradient descent (OGD) analysis techniques (see, e.g., \cite{hazan2022introduction}) to bound these terms. 


Term $(B)$ measures the tracking error between the actual system state $x_t$ and the target state $z_t$ . Under Assumption \ref{ass:strong_stab}, the system dynamics ensure exponential convergence of the actual state $x_t$ towards the target $z_t$. Additionally, the consecutive target states $z_t$ and $z_{t+1}$  generated by the algorithm remain close due to the gradient update step (Line \ref{lin:projection}), further controlling this tracking error. These two factors together allow us to bound these tracking error terms. 
The complete proof for this and other results can be found in Appendix \ref{app:proofs}.
\end{proof}

The benchmark employed in defining our regret $\cR_x(T)$ provides a valuable interpretation of the regret bound stated in Theorem \ref{thm:undist_regret_bound}. In particular, if the sequence of states $x_t$ were not governed by any dynamics but could instead be picked arbitrarily from the set $\cX$, our problem would reduce to an instance of the classical, `memory-less' online convex optimization problem. In that scenario, employing online gradient descent (OGD) would yield the \emph{optimal} regret bound of $1.5 LD^2 \sqrt{T}$ \cite{zinkevich_online_2003}. Thus, Theorem \ref{thm:undist_regret_bound} implies that the additional penalty (up to a constant additive term) we incur from learning in the presence of dynamics is captured through a multiplicative factor of  $\frac{4\sqrt{1+4\kappa^2}}{3\gamma}$.  
The constant additive term $\frac{\kappa LD^2}{\gamma}$ explicitly captures the small, yet non-zero, effect of the initial state of the system. 


\subsection{System with disturbance} \label{sec:disturbed_sys}
In the previous section, we considered the special case where the disturbances acting on our system (\ref{eq:LTI-sys-disturbed}) were absent. Let us now consider the setup when this is no longer the case and exploit the principle of superposition in LTI systems to re-write 
system (\ref{eq:LTI-sys-disturbed}) as
\begin{subequations}
    \begin{flalign}
        x_t & = \bar{x}_t + x_t^d \label{eq:superposition_x}\\
        \bar{x}_{t+1} & = A \bar{x}_t + B u_t \label{eq:superpos_undisturbed}\\
        x_{t+1}^d & = A x_t^d + w_t \label{eq:superpos_disturbance}
    \end{flalign}
\end{subequations}
where we assume $\bar{x}_1 = x_1$ and $x^d_1 = \mathbf{0}$. The `nominal' state $\bar{x}_t$ captures the state of the system without disturbances acting on it, while $x_t^d$ captures the state if it evolved in the presence of disturbances only. 

We can then also re-define the sequence of cost functions $\{f_t\}$ faced by the learner in terms of these nominal states as 
    \[ g_t(\bar x) \triangleq f_t(\bar x + x^d_t), \quad \forall \bar x \in \bR^n, \; \forall t.\]
This re-interpretation of the cost function formalizes the idea that both unknowns (the cost $f_t$ and the disturbances $\{w_\tau\}_{\tau=1}^t$), can be captured through a single unknown cost function $g_t$. 
Now, the problem we set out to solve -- the learner's \emph{original problem} -- is to control the disturbed system (\ref{eq:LTI-sys-disturbed}) when facing the cost sequence $\{f_t\}$. However, we can now define a related \emph{nominal problem}: controlling the undisturbed system (\ref{eq:superpos_undisturbed}) under the modified cost sequence $\{g_t\}$. The following proposition formalizes the equivalence of these two problems: 
\begin{prop} \label{prop:equiv_prob}
    For any sequence of inputs $\{u_t\}$, the regret incurred in the two problems are identical i.e. 
    \[\cR_u(T) = \cR_u^g(T) \triangleq \sum_{t=1}^T g_t(\bar{x}_t) - \min_{u\in \cU} \sum_{t=1}^T g_t(\bar{x}_t^u),\]
    where $\{\bar{x}_t^u\}$ denotes the trajectory of the system (\ref{eq:superpos_undisturbed}) under time-invariant input $u_t=u$.
\end{prop}

The proposition above follows from noting that $g_t(\bar x_t) = f_t(x_t)$ and $g_t(\bar x_t^u) = f_t( x_t^u)$ for all $t$. Proposition \ref{prop:equiv_prob} implies that if we design an online control policy that achieves sublinear regret for the \emph{nominal problem}, the resulting input sequence achieves good performance in the \emph{original problem} as well. Now, recall that this nominal problem - controlling a system without any disturbance when faced with adversarial costs - is precisely the setup we looked at in Section \ref{sec:undisturbed_sec}, and the online control policy (Algorithm \ref{alg:OLC}) we proposed required only the gradient of the cost functions as feedback in every step. Defining this gradient as $\bar{\delta}_t \triangleq \nabla g_t(\bar x_t)$, we can show that 
\vspace{-0.2cm}
\begin{flalign} \label{eq:equiv_grad}
    \bar{\delta}_t = \delta_t. \vspace{-0.1cm}
\end{flalign}
Since $g_t$ is obtained by composing $f_t$ with an affine transformation, (\ref{eq:equiv_grad}) follows from chain rule. With the description of the gradients of \emph{nominal} cost functions $\{ g_t\}$ offered by  (\ref{eq:equiv_grad}), we can state the main result of this section: 


\begin{thm} \label{thm:regret_bd_disturbed}
    Running Algorithm \ref{alg:OLC} on the nominal system with costs $\{g_t\}$ and gradients $\bar \delta_t$, and applying the resulting inputs
    $\{u_t\}$ to the original system yields the regret bound:
    \[\cR_u(T) \leq \frac{2LD^2}{\gamma}\bigg(\sqrt{T(1+4\kappa^2)}+2\kappa\bigg) \]
\end{thm}
 
While the regret bound in Theorem~\ref{thm:regret_bd_disturbed} does not explicitly depend on the magnitude of the disturbance, this dependence is implicitly captured in the constant $D$ which denotes a uniform bound on the system state under the dynamics~(\ref{eq:LTI-sys-disturbed}). This constant reflects the BIBO stability of the system and depends on the disturbance bound and initial conditions.

The algorithmic approach described in Theorem \ref{thm:regret_bd_disturbed} seems to suggest that we need to continually keep track of $x_t^d$ and generate the sequence of nominal costs $\{g_t\}$. But that is not the case. These nominal costs were introduced solely to motivate our solution approach which involved reducing the learning problem with adversarial disturbances to one without any. 
In practice, Algorithm~\ref{alg:OLC} can be run exactly as before: Equation~(\ref{eq:equiv_grad}) shows that the gradient used in the update step, $\nabla g_t(\bar x_t)$, is equal to the observed gradient $\nabla f_t(x_t)$, and 
 $g_t(\bar{x}_t) = f_t(x_t)$ as well. So, despite the reinterpretation, the algorithm requires no changes -- only the state update (Line~\ref{lin:x_update}) now evolves under disturbances according to~(\ref{eq:LTI-sys-disturbed}).




\section{Connection to Agarwal et al. } \label{sec:compare}
In this section, we briefly examine the connections and differences between our approach and that of Agarwal et al.~\cite{agarwal_online_2019} in addressing our problem. Their formulation is slightly more general in two respects: (i) they allow costs to depend explicitly on both the state and input, whereas we focus on dependence on state only, and (ii) they assume access to a strongly stabilizing controller $\cK$ (Definition 3.3, \cite{agarwal_online_2019}) instead of assuming that the system matrix $A$ is itself strongly stable. As noted in the introduction, our framework can be readily generalized to incorporate these aspects as well. Our assumptions are without much loss of generality and are made to facilitate a clear and concise presentation of our results. Accordingly, in drawing comparisons with their work, we will continue with the simpler setup used throughout this paper.


\subsection{Comparison of Policy Class} \label{para:policy-class-compare} The solution approach taken by Agarwal et al. \cite{agarwal_online_2019} considers inputs generated by the so called Disturbance Action Controller defined formally as:  
\begin{defn}[Disturbance Action Controller (DAC) \cite{agarwal_online_2019}] \label{defn:DAC}
Given a horizon $H$, sequence of past disturbances $\{w_\tau\}_{\tau<t}$, and sequence of matrices $M = \{M^{[0]},\dots,M^{[H-1]}\}$, a DAC chooses inputs
    \[u_t = \sum_{i=1}^H M^{[i-1]}w_{t-i}\quad \forall t\geq 0\]
where $w_t \triangleq \mathbf{0}$ for all $t\leq 0$. 
\end{defn}

The horizon length $H$ in the parametrization of such a controller serves two crucial purposes. One, it can be shown that the performance of every linear controller can be closely approximated by \emph{some} disturbance action controller, and that the accuracy of approximation increases with the length of this horizon (Lemma 5.2, \cite{agarwal_online_2019}). Two, having a finite `memory' horizon $H$ enables the application of results from Online Convex Optimization with Memory (OCO-M) \cite{Anava2015OnlineMistakes}. From this perspective, a shorter horizon is preferable, as the regret penalties associated with historical dependence scale with $H$. Consequently, the horizon length $H$ becomes a design parameter in the approach taken by Agarwal et al. \cite{agarwal_online_2019}, and is selected to be of order $\cO(\log T)$. 


In contrast, our method does not inherently rely on the framework of OCO-M, nor is it designed with the goal of approximating linear feedback controllers, which we argue to be somewhat ad-hoc as a performance benchmark for general convex costs. Hence, we do not require a controller with a fine-tuned memory horizon. 


\subsection{Comparison of Benchmark} At a first glance, the benchmark employed by Agarwal et al. \cite{agarwal_online_2019} appears to be the performance of the best stabilizing linear feedback controller. However, as pointed out in Section \ref{para:policy-class-compare}, it can be shown that the performance of any linear controller can be closely approximated by some \emph{disturbance action controller}, indicating that the DAC policy class is powerful enough to encompass the class of linear feedback controllers. Thus, the approach taken by Agarwal et al. \cite{agarwal_online_2019} achieves $\cO(\sqrt{T})$ regret against a stronger benchmark i.e. the performance of the best-in-hindsight \emph{direct action controller}. This benchmark can be described as the solution to the following optimization problem: 
\begin{flalign}
    \cB_M(T) \triangleq \max_{\substack{M \in \cM \\ u_t = \sum_{1}^H M^{[i-1]} w_{t-i}}} \; \sum_{t=1}^T g_t(\bar x_t) \label{eq:DAC-benchmark}
\end{flalign}
where the trajectory $\{\bar x_t\}_1^T$ is generated according to the dynamics in (\ref{eq:superpos_undisturbed}). Note that we are invoking Proposition \ref{prop:equiv_prob} in describing this benchmark in terms of the nominal costs $\{g_t\}$. 
While it is a bit difficult to analytically interpret the significance of this benchmark, the associated optimization problem is convex. In the following section, we present simulation results that aims to numerically capture its significance. 

As highlighted in (\ref{eq:reg-defn-u}), the benchmark in our work is the performance of the best time-invariant input. Once again using Proposition \ref{prop:equiv_prob}, we can describe this benchmark in terms of the nominal costs as
\begin{flalign} \label{eq:benchmark_u}
    \cB_u(T) \triangleq \max_{\substack{u \in \cU}} \; \sum_{t=1}^T g_t(\bar x_t^u)
\end{flalign} 
We also saw the following equivalence:
\begin{flalign*} 
      \cB_u(T) \stackrel{\text{Lemma \ref{lem:equiv-regret}}}{=} \max_{\substack{\bar x \in \cX }} \; \sum_{t=1}^T g_t(\bar x) + o(T)
\end{flalign*}
which provides an alternate interpretation for our benchmark, namely, ``given complete knowledge about the sequence of costs and the disturbances the learner will face, what `nominal' state should they aim to stabilize at". In the following section, we will numerically evaluate this benchmark as well and compare it against the one in (\ref{eq:DAC-benchmark}). 

\section{Simulations} \label{sec:simu}

We run all our simulations on the system with state and input matrices \vspace{-0.2cm}
\[A = \frac{1}{3.6}\begin{bmatrix}
    1 & 0.2 & 0 \\
    0 &  1 &  0.2 \\
    0.2 & 0 & 1 \\
\end{bmatrix} \text{ and } B = \begin{bmatrix}
    0 & 1 \\
    0 & 0 \\
    1 & 0
\end{bmatrix}.\]

\begin{figure}[ht]
     \centering
     \begin{subfigure}[b]{0.45\textwidth}
         \centering
         \includegraphics[width=\textwidth]{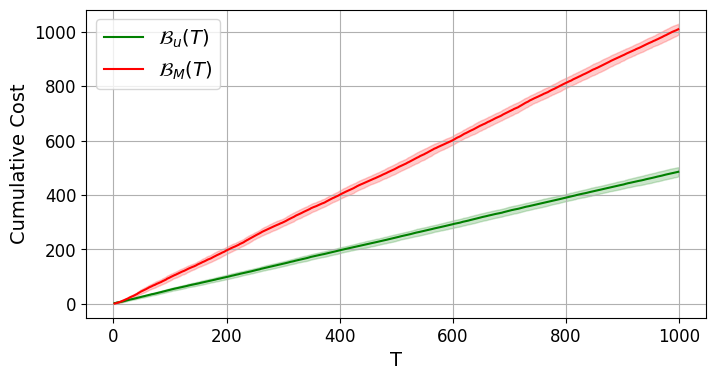}
         \caption{Comparison of benchmarks}
         \label{fig:bench_compare}
     \end{subfigure}
     \begin{subfigure}[b]{0.45\textwidth}
         \centering
         \includegraphics[width=\textwidth]{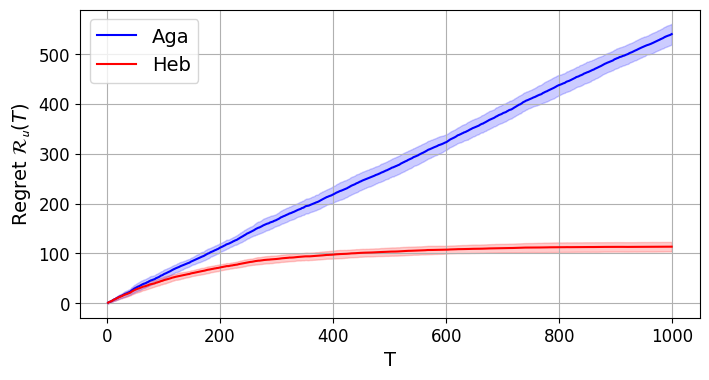}
         \caption{Regret against fixed-input Benchmark $\cB_u(T)$}
         \label{fig:regret_FI}
     \end{subfigure}
     \begin{subfigure}[b]{0.45\textwidth}
         \centering
         \includegraphics[width=\textwidth]{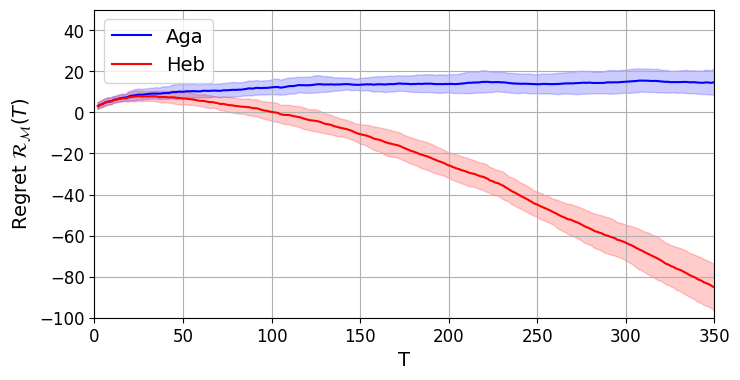}
         \caption{Regret against DAC benchmark $\cB_M(T)$}
         \label{fig:regret_DAC}
     \end{subfigure}
        \caption{Results of numerical simulations} 
        \label{fig:simu-results}
\end{figure}

The cost functions take the form 
\[f_t(x) = (x-c_t)^\intercal Q_t (x - c_t) \quad \forall \; 1\leq t \leq T\]
where the sequence of positive definite matrices $\{Q_t\}$ and `target' vectors $\{c_t\}$ are generated i.i.d from their respective distribution\footnote{For code and more details - github.com/vijeth27/onlineLinearControl.git}. The disturbances $\{w_t\}$ are also sampled i.i.d from a uniform distribution. 

We run both the proposed method (Algorithm \ref{alg:OLC}), referred to as \textbf{Heb}, and the online control method of Agarwal et al. \cite{agarwal_online_2019}, referred to as \textbf{Aga} for $T=1000$ time steps on the problem outlined above. We run $N=20$ iterations each with a newly sampled sequence of $\{Q_t\},\{c_t\}$ and $\{w_t\}$. 

Figure \ref{fig:bench_compare} compares the two benchmarks -- the cumulative cost $\cB_M(T)$ of the best-in-hindsight DAC policy, and the cumulative cost $\cB_u(T)$ of the best-in-hindsight fixed input. We observe that the best DAC policy is consistently outperformed by the best fixed input, suggesting that the latter can serve as a better point of comparison in certain online control scenarios.  

Figures \ref{fig:regret_FI} and \ref{fig:regret_DAC} depict the regret performance of both algorithms with respect to the two benchmarks described. Against the fixed input benchmark $\cB_u(T)$, our method \textbf{Heb} demonstrates sublinear regret growth, aligning with the theoretical guarantees presented in Theorem \ref{thm:regret_bd_disturbed}. In contrast, \textbf{Aga} exhibits significantly poorer performance. When we consider the DAC policy benchmark $\cB_M(T)$, \textbf{Aga} achieves sublinear regret consistent with the theoretical results of Agarwal et al. \cite{agarwal_online_2019}. On the other hand, \textbf{Heb} rapidly outperforms this benchmark, gradually incurring negative regret as the horizon increases.




\section{CONCLUSIONS AND FUTURE WORKS}

\paragraph{Conclusions} In this paper, we proposed a novel benchmark for evaluating online non-stochastic control methods: the performance of the best-in-hindsight fixed input. We showed that, in the absence of disturbances, this benchmark captures the state at which the system would stabilize to incur minimal cumulative cost. Thus, our proposed benchmark serves as a natural extension of the standard benchmark commonly used to define static regret in online convex optimization. Through numerical simulations, we also demonstrated that our proposed benchmark outperforms the Disturbance Action Controller (DAC) policy benchmark. By leveraging the equivalence established by Agarwal et al. \cite{agarwal_online_2019}, this also implies that our approach can surpass the performance of the best linear controller, the previously established benchmark for online non-stochastic control.

We also introduced a novel online control algorithm that achieves sublinear regret relative to our proposed benchmark. 
We presented theoretical guarantees on the performance of this algorithm and validated these guarantees through numerical simulations. Thus, our work advances the theory of online non-stochastic control by introducing a novel, conceptually meaningful performance metric and proposing an effective online control method to achieve competitive performance in online non-stochastic control scenarios.

\paragraph{Future Works} We are currently investigating extensions of our approach that impose additional structure on the cost functions, such as strong convexity, with the aim of obtaining logarithmic regret bounds. Another direction involves extending our framework to more realistic scenarios, including settings with safety constraints on the states or incomplete information about the system dynamics.

While our paper assumes a stable system matrix, our method can also be applied under the weaker assumption that the pair $(A,B)$ is merely stabilizable. In that case, however, we also require full observability of the system without which we cannot design a stabilizing linear feedback controller. Relaxing this assumption and considering setups with partial observability is another challenging avenue for further research.




\addtolength{\textheight}{-0cm}   


\bibliography{onlineLinCont}
\bibliographystyle{ieeetr}

\appendices

\section{Extensions to input dependent costs} \label{app:input-dependent}
We briefly sketch how our approach can be extended to settings where the cost functions depend explicitly on both the state and the input. That is, we consider cost functions of the form \[f_t:(x,u) \in \bR^N \times \bR^M \mapsto f_t(x,u) \in \bR.\] As before, we first consider the system without any disturbances. After incurring the cost $f_t(x_t,u_t)$ at time $t$, the learner observes the gradient $(\delta_t^x,\delta_t^u) = \big(\nabla_x f_t(x_t,u_t),\nabla_x f_t(x_t,u_t)\big).$ We redefine the steady state manifold (Definition \ref{def:ss_set}) to account for both the state and the corresponding input that stabilizes the system at this state: 
\[\cX_{u} \triangleq \big\{(x,u) \in \bR^N \times \cU | x = A x + B u \big\}\]
We can now define a natural extension of our notions of regret in (\ref{eq:regret_u_intro}) and (\ref{eq:regret_x_intro}) as 
\begin{flalign}
    \cR_{xu}(T) \triangleq \sum_{t=1}^T f_t(x_t,u_t) - \min_{(x,u) \in \cX_u^* } \sum_{t=1}^T f_t\big( x , u \big). \label{eq:regret_xu}
\end{flalign}
To minimize this regret, we propose a straightforward extension of Algorithm \ref{alg:OLC}. We initialize the iterates as $(z_1,x_1) \in \cX_u$ and replace the steps in Lines \ref{lin:invert_z} and \ref{lin:projection} with the update
\begin{flalign}
    (z_{t+1},u_{t+1} )& = \Pi_{\cX^*_u} \big( (z^t - \eta \delta_t^x,u^t - \eta \delta_t^u)  \big). 
\end{flalign}
Following the regret analysis used in Theorem \ref{thm:undist_regret_bound}, we can show that this modified algorithm yields sublinear regret with respect to \eqref{eq:regret_xu}. Specifically, we have 
\[\cR_{xu}(T) \in \cO\bigg(\frac{\kappa (L_u + L_x) D^2 \sqrt{T}}{\gamma}\bigg),\]
where $L_u$ and $L_x$ denote the smoothness parameters for the cost functions with respect to state and input, respectively. The extension of this approach to systems with adversarially generated disturbances follows the approach presented in Section \ref{sec:disturbed_sys}.


\section{Proofs} \label{app:proofs}
\subsection{Proof of Lemma \ref{lem:equiv-regret}}
Let $x^*\in \cX$ denote the state at which the cumulative cost $\sum_{t=1}^T f_t(x)$ is minimized and let $u^*$ be the corresponding steady-state input. Let $\{x_t^*\}$ denote the state trajectory starting from initial state $x_1$ under the time-invariant input $u_t=u^*$. Then, we have

\begin{flalign*}
    \cR(T) - \cR^u(T) & = \min_{u\in \cU} \sum_{t=1}^T f_t(x_t^u)- f_t(x^*) \\
    & \stackrel{}{\leq} \sum_{t=1}^T f_t(x_t^*) - f_t(x^*)
\end{flalign*}
where the final inequality results from the definition of $x^*_t$. The convexity and smoothness of $f$ gives us 
    \begin{flalign}
        \sum_{t=1}^T f_t(x^*_t) - f_t(x^*) \leq LD 
        \sum_{t=1}^T \|x^* - x^*_t\|. \label{eq:const-inp-proof1}
    \end{flalign} 
    Then, defining the error $e_t = x^* - x^*_t$, the dynamics under constant input gives us
    \begin{flalign*}
        e_{t+1} = A x^* + Bu^* - A x_t^* - Bu^* = A e_t = A^{t} e_1  
    \end{flalign*}
    where the final step follows by induction. The norm of error terms can then be bound as 
    \begin{flalign*}
        \|e_{t+1}\| \leq \|A^{t}\| \|e_1\| \leq \kappa (1-\gamma)^{t} \|e_1\|.  
    \end{flalign*}
    Our result then follows from substituting this norm bound back in (\ref{eq:const-inp-proof1}), from noting that \[\sum_{1}^T (1-\gamma)^{t-1} \leq \sum_0^\infty (1-\gamma)^t = \frac{1}{\gamma}.\]
    and by observing that $\|e_1\| \leq \|x^*\| + \|x_1^*\| \leq 2 D$. 
    \hfill $\blacksquare$
\subsection{Proof of Theorem \ref{thm:undist_regret_bound}}
We first present some lemmas to aid us in the regret analysis.

\begin{lem}(Neumann Series) \label{lem:neumann}
    Let $A\in \bC^{n\times n}$ with spectral radius $\rho(A)<1$. Then the Neumann series of $A$ converges and $\sum_{0}^\infty A^k = (I-A)^{-1}.$ 
\end{lem}

\begin{lem} \label{lem:diff-z}
    For any $t\geq 1$ we have, \vspace{-0.2cm}
    \[\|z^{t+\tau} - z^t\| \leq \eta \tau L \quad \forall \tau \geq 1.\]
\end{lem}
\begin{proof} For $\tau=1$ we have  
    \begin{flalign*}
        \|z^{t+1} - z^t\| = & \|\Pi_{\cX}(z^t - \eta \delta_t) - z^t\| \\
        \stackrel{(a)}{\leq} & \|(z^t - \eta \delta_t) - z^t\| \stackrel{(b)}{\leq } \eta L D
    \end{flalign*}
    where $(a)$ follows from contractivity of $\Pi_\cX$ and $(b)$ follows from Assumption \ref{ass:smooth-cost}. Suppose the proposition holds for some $\tau\geq 1$, we will use a similar analysis to show it also holds for $\tau+1$. We have
    \begin{flalign*}
        \|z^{t+\tau+1} - z^t\| = & \|\Pi_{\cX}(z^{t+\tau} - \eta \delta_{t+\tau}) - z^t\| \\
        \stackrel{(a)}{\leq} &\|(z^{t+\tau} - \eta \delta_{t+\tau}) - z^t\| \\ \stackrel{}{\leq}  & \|z^{t+\tau} - z^t\| + \eta \|\delta_t\| \stackrel{(b)}{\leq} (\tau+1) \eta L D
    \end{flalign*}
    where $(a)$ results from the contractivity of $\Pi_\cX$, and $(b)$ results from a combination of the inductive hypothesis and Assumption \ref{ass:smooth-cost}.
\end{proof}
Equipped with these lemmas, let us now continue with obtaining our regret bound.

\begin{proof} 
The convexity of $f_t$'s gives us
\begin{flalign}
    Regret(T) & = \sum_{t=1}^T f_t(x_t) - f_t(x^*) \leq  \sum_{t=1}^T \delta_t^\intercal (x_t - x^*) \notag \\
    & = \sum_{t=1}^T \underbrace{\delta_t^\intercal (z_t - x^*)}_{(A)} + \underbrace{\delta_t^\intercal (x_t - z_t)}_{(B)} 
\end{flalign}
We will use two different analytic techniques to bound the two types of terms $(A)$ and $(B)$. 

Our approach to bounding terms of the type $(A)$ closely follows the regret analysis technique \cite{hazan2022introduction} for the online gradient descent (OGD) learning algorithm. To begin, consider
\begin{flalign*}
    \|z_{t+1} - x^* \|^2 & = \big\|\Pi_{\cX}(z_t -\eta \delta_t) - x^* \big\|^2 \stackrel{(a)}{\leq}  \|z_t -\eta \delta_t - x^* \|^2 \\
    & = \|z_t - x^*\|^2 - 2 \eta \delta_t^\intercal (z_t-x^*) + \eta^2 \|\delta_t\|^2 
\end{flalign*}
where $(a)$ results from noting that $x^*\in \cX$ and from the contractivity of projections over convex sets. Re-arranging terms and recalling Assumption \ref{ass:smooth-cost} gives us 
\begin{flalign*}
    \delta_t^\intercal (z_t-x^*) &\leq  \frac{1}{2\eta}\big(\|z_t - x^*\|^2 -  \|z_{t+1} - x^*\|^2\big) + \frac{\eta}{2} L^2 D^2.
\end{flalign*}

Shifting our focus to terms of the form in $(B)$, we have
\begin{flalign*}
    & x_t - z_t = A^{t-1} x_1 + \sum_{\tau=1}^{t-1} A^{\tau-1} B u_{t-\tau} - z_t \\ 
    & \stackrel{(b)}{=} A^{t-1} x_1 + \sum_{\tau=1}^{t-1} A^{\tau-1} (I-A) z_{t-\tau} - z_t \\
    & \stackrel{(c)}{=} A^{t-1} x_1 + \sum_{\tau=1}^{t-1} A^{\tau-1} (I-A) z_{t-\tau} - \sum_{\tau=0}^\infty A^\tau (I-A) z_t \\
    & = A^{t-1} x_1 + \sum_{\tau=1}^{t-1} A^{\tau-1} (I-A) \big(z_{t-\tau}-z_t\big) \\ 
    & - A^{t-1} \sum_{\tau=0}^\infty A^\tau (I-A) z_t \\
    & \stackrel{(d)}{=}  A^{t-1} \big(x_1-z_t\big) + \sum_{\tau=1}^{t-1} A^{\tau-1} (I-A) \big(z_{t-\tau}-z_t\big)
\end{flalign*}
where $(b)$ follows from Line \ref{lin:invert_z} of Algorithm \ref{alg:OLC} and $(c,d)$ follows from noting that $\rho(A) = \rho(L) < 1$ and invoking Lemma \ref{lem:neumann}. Taking norms on both sides and employing triangle inequality and properties of matrix norms gives us
\begin{flalign*}
    & \|x_t - z_t\| \\
    & \leq  \|A^{t-1}\| \|x_1-z_t\| + \sum_{\tau=1}^{t-1} \|A^{\tau-1}\| \|I-A\| \|z_{t-\tau}-z_t\| \\
    & \stackrel{(e)}{\leq} \kappa (1-\gamma)^{t-1} \big(\|x_1\| + \|z_t \|\big) +\sum_{\tau=1}^{t-1} 
 2 \kappa^2  (1-\gamma)^{\tau-1}\eta \tau LD   \\
    & \stackrel{(f)}{\leq} 2 D \kappa (1-\gamma)^{t-1} + 2 \eta \kappa^2 LD \sum_{\tau=1}^{\infty} 
(1-\gamma)^{\tau-1}\tau \\
    & \stackrel{(g)}{=} 2 D \kappa (1-\gamma)^{t-1} + 2 \frac{\eta \kappa^2 LD}{\gamma^2} 
\end{flalign*}
where $(e)$ follows from the fact that 
\begin{flalign*}
    \|A^k\| & = \|H^{-1} J^{k} H\| \leq \|H\|\|H^{-1}\| \|J\|^k, \\
    \|I-A\| & = \|H^{-1} (I -J) H\| \leq \kappa \|I-J\| \leq 2 \kappa,
\end{flalign*}
and from Lemma \ref{lem:diff-z}. Inequality $(f)$ follows from the bound on the states of the system. Note that $z_t$ is the steady state of the system under constant input $u_t$. Equality $(g)$ results from summing up an AGP series. 

Let us now continue our regret analysis from (\ref{eq:regret_analysis_1}) using the bounds we obtained for terms $(A)$ and $(B)$. Using Cauchy-Schwarz and Assumption \ref{ass:smooth-cost} we have   
    \begin{flalign*}
        Regret(T)  \leq & \sum_{t=1}^T \textcolor{red}{\delta_t^\intercal (z_t - x^*)} + \textcolor{blue}{LD \|x_t - z_t\|} \\
        \leq &\sum_{t=1}^T \bigg( \textcolor{red}{\frac{1}{2\eta}\big(\|z_t - x^*\|^2 -  \|z_{t+1} - x^*\|^2\big)} \\
        + & \textcolor{red}{\frac{\eta}{2} L^2 D^2} + \textcolor{blue}{2 LD^2 \kappa (1-\gamma)^{t-1} + \frac{2\eta \kappa^2 L^2D^2}{\gamma^2}} \bigg) \\
        = &\textcolor{black}{\frac{1}{2\eta} \|z_1 - x^*\|^2} + \eta L^2 D^2 T \bigg(\textcolor{black}{\frac{1}{2}} + \textcolor{black}{\frac{2\kappa^2}{\gamma^2}}\bigg) \\
        & +  \textcolor{black}{2 L D^2 \kappa \sum_{t=0}^{T-1} (1-\gamma)^t} \\
        \stackrel{(h)}{\leq} \; & 2\frac{D^2}{\eta} + \eta L^2 D^2 T \frac{1+4\kappa^2}{2\gamma^2} +\frac{2LD^2 \kappa}{\gamma}.  
    \end{flalign*} 
    where inequality $(h)$ results from the bound on the state of the system and from noting that $\gamma\in (0,1]$. Setting $\eta = \frac{2\gamma}{L\sqrt{T(1+4\kappa^2)}}$ gives us the regret bound 
    \begin{flalign*}
        Regret(T)  \leq \frac{2 L D^2 \sqrt{T(1+4\kappa^2)}}{\gamma} +\frac{2LD^2 \kappa}{\gamma}.  
    \end{flalign*} 
    \end{proof}
\subsection{Proof of Theorem \ref{thm:regret_bd_disturbed}}
From (\ref{eq:equiv_grad}) and Assumption \ref{ass:smooth-cost} we see that the gradient of the nominal cost function satisfies
\[\|\nabla g_t(\bar x_t)\| = \|\nabla f_t(x_t)\| \leq LD.\]
 We also observe that $\|\bar x_t\|<D$ because of the BIBO stability of (\ref{eq:superpos_undisturbed}). So we see that $g_t$ and $\bar x_t$ satisfy all the assumptions we made on $f_t$ and $x_t$ respectively. So by running Algorithm \ref{alg:OLC} with the sequence of costs $g_t$ and system (\ref{eq:superpos_undisturbed}) allows us to invoke Theorem \ref{thm:undist_regret_bound} and state 

 \[\sum_{t=1}^T g_t(\bar{x}_t) - \min_{\bar x\in \cX} \sum_{t=1}^T g_t(\bar{x}) \leq \frac{2LD^2}{\gamma}\bigg(\sqrt{T(1+4\kappa^2)}+\kappa\bigg).\]
 The result then follows from invoking Lemma \ref{lem:equiv-regret} and from Proposition \ref{prop:equiv_prob}.

\end{document}